\documentclass[11pt]{amsart}
\usepackage{amsfonts}
\usepackage{amsmath}
\usepackage{amsthm}
\usepackage{bm,bbm}
\usepackage[dvips]{graphicx}
\usepackage{caption}
\usepackage[round]{natbib}
\usepackage{color}
\usepackage[normalem]{ulem}
\usepackage{cancel}

\usepackage{bigints}

\usepackage{bm}

\numberwithin{equation}{section}

\allowdisplaybreaks

\theoremstyle{plain}
\newtheorem{theorem}{Theorem}[section]
\newtheorem{lemma}[theorem]{Lemma}
\newtheorem{proposition}[theorem]{Proposition}

\theoremstyle{definition}

\newtheorem{remark}[theorem]{Remark}

\def\beqn{\begin{equation}}
\def\beqn*{$$}
\def\eeqn{\end{equation}}

\def\P{\mathbb{P}}

\def\E{\mathbb{E}}

\newcommand{\bbn}{{\mathbb N}}

\newcommand{\vep}{\varepsilon}

\newcommand{\one}{{\mathbbm 1}}

\newcommand{\remove}[1]{}

\newcommand{\bp}{{\bf p}}

\newcommand{\calW}{\mathcal{W}}
\newcommand{\calH}{\mathcal{H}}

\begin{document}

\bibliographystyle{plainnat}

\title[Large deviations for subcomplex counts]
{Large deviations for subcomplex counts and Betti numbers in multi-parameter simplicial complexes}
\author{Gennady Samorodnitsky}
\address{School of Operations Research and Information Engineering\\
Cornell University \\
NY, 14853, USA}
\email{gs18@cornell.edu}
\author{Takashi Owada}
\address{Department of Statistics\\
Purdue University \\
IN, 47907, USA}
\email{owada@purdue.edu}

\thanks{Samorodnitsky's research is partially supported by the
  NSF grant DMS-2015242  at  Cornell  University. Owada's research is partially supported by the NSF grant, 
  DMS-1811428 at Purdue University. }


\begin{abstract}
 
  \noindent
We consider the multi-parameter random simplicial complex as a higher
dimensional extension of the classical Erd\"os-R\'enyi graph. We
investigate appearance of  ``unusual" topological structures in the
complex from the point of view of large deviations. We first study upper tail large deviation probabilities for subcomplex counts, deriving the order of magnitude of such probabilities at the logarithmic scale precision. The obtained results are then applied to analyze large deviations for the number of simplices at the critical dimension and below. Finally, these results are also used to deduce large deviation estimates for Betti numbers of the complex in the critical dimension. 
 
\end{abstract}

\maketitle

\section{Introduction}  \label{s:introduction}
One can view a simplicial complex as a network with connections
potentially 
involving more than 2 vertices at a time. Given a set $V$ of vertices,
an undirected 
graph allows for the existence of only edges of the type $(v_1,v_2)$ for
$v_1,v_2$ in $V$, while potentially higher-dimensional ``edges'' would
have the form $(v_1,\ldots, v_k)$ with $k\geq 2$ for $v_1,\ldots, v_k$
in $V$. If $k>2$, this is a hyperedge, which is not allowed in a graph
but is allowed in a hypergraph. A simplicial complex is a
special kind of a hypergraph, in which a subset of hyperdge is itself
a hyperedge. That is, if $(v_1,\ldots, v_k)$ with $k> 2$ is a
hyperedge, then so is 
the collection of $k-1$ vertices obtained by removing from
$(v_1,\ldots, v_k)$ any one of its $k$ vertices. When describing a
simiplical complex, one typically says that $(v_1,\ldots, v_k)$ forms
a $(k-1)$-dimensional simplex (henceforth we call it a $(k-1)$-simplex), and not a hyperedge.

In a random simplicial complex, the simplices of different dimensions are
added according to a randomized rule. Some of the models of random
simplicial complexes are 
extensions of the classical Erd\"os-R\'enyi random graph, in which  
potential edges between two vertices are formed with  probability $p$, 
independently of other potential edges. Possible rules of constructing
a random simplicial complex include the \emph{flag complex} (also known as
the clique complex), in which a potential $k$-simplex is
formed whenever a set of $k+1$ vertices constitutes a  clique in the
Erd\"os-R\'enyi graph (see e.g., \cite{kahle:2009}).   
\emph{The Linial-Meshulam-Wallach complex} of a fixed maximal dimension $k$
is a random simplicial complex, in which all of the $(k-1)$-simplices are  present with probability $1$, while the potential $k$-simplices are included with probability $p$, independently of other $k$-simplices (\cite{linial:meshulam:2006, meshulam:wallach:2009}).  
The most general model in this direction is 
\emph{the Costa-Farber multi-parameter simplicial complex}, for which
potential simplices are added inductively in their dimensions; for
every $k=1,2,\dots$, each potential $k$-simplex is included with
probability $p_k$, independently of other simplices, only when all of its $(k-1)$-faces are present (see
\cite{costa:farber:2016, costa:farber:2017}). 

The randomness of the  simplicial complexes induces
randomness on the topological structure of the complex and its
topological invariants, such as 
\emph{the Betti numbers} and \emph{the Euler characteristic}. The distributions of
 topological invariants have been a subject of recent interest for
various models of random simplicial complexes. This includes the existence
of a dominating dimension and central limit theorems for the Euler
characteristic; see e.g. \cite{kahle:meckes:2013,
  thoppe:yogeshwaran:adler:2016, fowler:2019, kanazawa:2022}. Functional
limit theorems for a dynamic version of the multi-parameter model have
been established in \cite{owada:samorodnitsky:thoppe:2021}.

However, the previous work   describes only the ``usual" topological
structure of  random simplicial complexes, in the sense of the
``average" behavior  and likely deviations from the ``average"
behavior of the topological invariants. 
In contrast, the  primary focus of this paper is in the situations when the
topological structure of the complex is less usual, in the sense of
a topological invariant being far away from the average. Such events
are, by definition, rare but may have an oversized impact on the
function of the network, and they are typically referred to as large
deviations events. 
Such events are often related to situations when 
 certain subcomplexes   appear significantly more or significantly
 less than expected. Understanding the
probabilities of such events is sometimes described as the upper (or
lower) tail large deviations problem for a
subcomplex count.


Within the context of the Erd\"os-R\'enyi random graphs, 
large deviation problems for subgraph counts have attracted much attention over the last decade in, among many others, \cite{chatterjee:varadhan:2011, chatterjee:dembo:2016, bhattacharya:ganguly:lubetzky:zhao:2017, lubetzky:zhao:2017, eldan:2018, yan:2020, cook:dembo:2020}. 
In particular, \cite{chatterjee:2017} gives a comprehensive presentation, covering various large deviation problems for the random graph of different degrees of denseness. Moreover, \cite{janson:oleszkiewicz:rucinski:2004} developed a general
framework for  the upper tail large deviation problems for subgraph counts in the random graph. 
The combinatorial part of this approach is based on a result in \cite{alon:1981}, and an extension to uniform subhypergraph counts in a random hypergraph setup is provided in \cite{dudek:polcyn:rucinski:2010}. 

The present work receives much inspiration from \cite{janson:oleszkiewicz:rucinski:2004} and addresses the upper tail large deviation problems for subcomplex counts in the multi-parameter random simplicial complex. 
Due to its high-dimensional topological structure, 
this problem is more involved than  the analogous problems for
random graphs. 
Although some of the results available for the Erd\"os-R\'enyi random
graph have not yet been fully extended to the multi-parameter random
simplicial complex, the results we obtain  are useful for
understanding certain rare events in the latter 
complex. For instance, we are able to describe  the upper tail large
deviations for the Betti numbers at the dominating dimensions.  

This paper is organized as follows. In Section \ref{sec:basics}, we
present a formal definition of  the multi-parameter random simplicial
complex.  
In Section \ref{sec:upper.tails}, we address 
general upper tail large deviation problem for subcomplex
counts. Section \ref{sec:simplices}  specializes to the large
deviation problem for the number of simplices at the
dominating dimension and below, under the setup of
\cite{owada:samorodnitsky:thoppe:2021}. Finally, Section
\ref{sec:betti} discusses the upper tail large deviations for  Betti numbers
at the dominating dimension.

\section{The multi-parameter random simplical complex and the  
  subcomplex count problem} \label{sec:basics}

In this section we formally construct the multi-parameter random
simplical complex introduced in \cite{costa:farber:2016,
  costa:farber:2017}. This complex  is a model of an abstract
simplicial complex on the alphabet $[n]=\{1, \ldots, n\}$,
parametrized by $\bp=\bp(n)=(p_1,\ldots, p_{n-1})\in [0,1]^{n-1}$ . 
A simplex
(or a face, or a word) in this complex is a nonempty collection of letters in
the alphabet, and the dimension of a simplex is equal to the number of
letters in the word minus 1. The $i$th skeleton of a complex is the 
subcomplex consisting of all faces of dimension $i$ or less. The
multi-parameter random simplicial complex is built recursively,
starting with the $[n]$ as the 0th skeleton. For $i=1,\ldots,
n-1$, once the $(i-1)$st
skeleton has been constructed, each of the potential $i$-simplices whose
boundary is in that $(i-1)$st skeleton, is added to the complex with
probability $p_i$, independently of other potential $i$-simplices. We
denote the obtained random complex by $K\big(n; \bp(n)\big)=K(n;p_1,\ldots, p_{n-1})$. 

We are interested in the subcomplex counting problem for the
multi-parameter random simplicial complex.
Given two simplicial complexes, $F_1$ and $F_2$, an
ordered copy of $F_1$ in $F_2$ is an injective simplicial map from
the vertex set of $F_1$ to the vertex set of $F_2$. In particular, this map has the 
property that the vertices of every simplex in $F_1$ are mapped
into the vertices of a simplex in $F_2$ of the same dimension. Similarly, an unordered
copy of $F_1$ in $F_2$ is a subcomplex $F_3$ of
$F_2$, which is isomorphic to $F_1$; that is, there is  a bijective mapping between 
the vertex set of $F_1$ and the vertex set  of $F_3$, such that a set of vertices forms a simplex in $F_1$  if and only if the
corresponding vertex set under this mapping forms a simplex in $F_3$. If
$n_o(F_2,F_1)$ and $n(F_2,F_1)$ are the numbers of
ordered and unordered copies of $F_1$ in $F_2$
correspondingly, then it is
clear that $n_o(F_2,F_1)/n(F_2,F_1)=\#({\rm Aut}(F_1))$,
the number of automorphisms of $F_1$ consisting of all permutations of
the vertices of $F_1$ preserving the complex. 

Let $F$ be a fixed simplicial complex of dimension $k\ge1$. For
$n>k$ and probabilities $\bp(n)$, we denote by $N_n(F)$ and $N_{o,n}(F)
$ the (random)
numbers of unordered and ordered copies of $F$ in the multi-parameter simplicial complex
$K(n;p_1,\ldots, p_{n-1})$, correspondingly.  Note that $p_i$ with $i>k$ do not affect
$N_n(F)$ and $N_{o,n}(F)$, so we may assume that
$p_i=0$ for $i>k$ and simply write the complex as $K(n;p_1,\ldots, p_k)$. If we
denote for $i=0,1,\ldots, k$, 
\begin{align} \label{e:number.simp}
  F_i =& \ \text{the set of $i$-simplices in $F$,} \\
 \notag s_i(F) =& \ \text{the number of $i$-simplices in $F$}, 
\end{align}
then
\begin{align} \label{e:mean.ordered}
&\mu_{o,n}(F) := \E [N_{o,n}(F)]= (n)_{s_0(F)}\prod_{i=1}^k
                                       p_i^{s_i(F)}, \\
 &\mu_{n}(F) := \E [N_{n}(F)]=
\mu_{o,n}(F)/\#({\rm Aut}(F)),  \notag 
\end{align}
where $(n)_{s_0(F)}:=n(n-1)\cdots \bigl( n-s_0(F)+1\bigr)$. 

As $n\to\infty$, the copies of $F$ can potentially be found in many
(nearly) independent parts of the multi-parameter simplicial complex
$K(n;p_1,\ldots, p_k)$, so one expects 
that for large $n$ enough, $N_n(F)$ and $N_{o,n}(F)$ do not deviate
``too much'' from their corresponding means. Therefore, the upper tail
large deviation probabilities
\begin{equation}
\label{e:uppertail.ld}
\P\bigl(N_{o,n}(F)\ge (1+\vep) \mu_{o,n}(F)\bigr), \ \ \vep>0, 
\end{equation}
and the lower tail large deviation probabilities 
\begin{equation}
\label{e:lowertail.ld}
\P\bigl(N_{o,n}(F)\le (1-\vep) \mu_{o,n}(F)\bigr), \ \ 0<\vep<1, 
\end{equation}
are expected to be exponentially small for large  $n$ enough. Our
subject of interest is to investigate  exactly how small these
probabilities are. In this paper we focus only on the upper tail large
deviations in \eqref{e:uppertail.ld}. An analysis of the lower tail
large deviations in \eqref{e:lowertail.ld} is postponed to a future
publication.  

\section{Upper tail large deviations} \label{sec:upper.tails}

Our approach to understanding the upper tail large deviations for
subcomplex counts is inspired by
\cite{janson:oleszkiewicz:rucinski:2004}.  
If $G$ is a fixed simplicial complex of dimension $k\ge1$, we denote by
$N(m_0,m_1,\ldots, m_k; G)$ the maximum of $n(F,G)$ taken over all
simplicial complexes $F$ with   $s_i(F)\leq m_i$, 
$i=0,1,\ldots, k$. Clearly, $N(m_0,m_1,\ldots, m_k; G)=0$ unless
$s_i(G)\leq m_i$, $i=0,1,\ldots, k$.

The number $N(m_0,m_1,\ldots, m_k; G)$ is often referred to as the \emph{extremal parameter} and is related  to a certain
linear optimization problem that we now describe. Using the notation
in \eqref{e:number.simp}, we consider the linear program
\begin{align} \label{e:linear.prog}
  &\max \sum_{v\in G_0}x_v \\
  &\text{subject to} \notag \\ 
    &0\le \sum_{v\in \sigma_i}x_v\leq \log m_i, \  \sigma_i\in G_i, \
    i=0,1,\ldots, k. \notag 
\end{align}
Denote by $\gamma=\gamma(m_0,m_1,\ldots, m_k;G)$ the optimal value of
this problem.

\begin{proposition} \label{prop:N.LP}
  Assume that $s_i(G)\leq m_i$, $i=0,1,\ldots, k$. Then, there are finite
  positive constants $c(G),C(G)$ that depend only on $G$, such that 
  \begin{equation} \label{e:both.b.N}
 c(G)e^{\gamma(m_0,m_1,\ldots, m_k;G)} \leq N(m_0,m_1,\ldots, m_k;G)\leq
 C(G)e^{\gamma(m_0,m_1,\ldots, m_k;G)}.  
\end{equation}
 \end{proposition} 
\begin{proof}
We first prove the lower bound in \eqref{e:both.b.N}.
Let $\bigl( x_v^*, \, v\in G_0\bigr)$ be an optimal solution to the
linear program \eqref{e:linear.prog}. 
 We construct a simplicial complex $F$ as
follows. Let $c>0$ be a small constant described in the sequel. 
We start with a family of disjoint sets $(V_v)_{v\in G_0}$, where $V_v$ consists of $n_v:=\lceil ce^{x_v^*}\rceil$ points for each $v\in G_0$.  
Define the vertices of $F$ to be the points in 
 the union $\bigcup_{v\in G_0}V_v$, i.e., we take $F_0=\bigcup_{v\in G_0}V_v$. 
Next, for every $j\in \{ 1,\dots,k \}$ and distinct vertices $v_1, \dots, v_{j+1}\in G_0$, a point set $(w_1,\dots,w_{j+1})\in \prod_{i=1}^{j+1}V_{v_i}$ forms a $j$-simplex in $F$ if and only if  
 the vertices
$v_i, \, i=1,\ldots, j+1$, form a $j$-simplex in $G$. 

We claim that if we choose a sufficiently small constant $c>0$ that depends only on $G$, 
 then it can be assured that
\begin{equation*}
s_j(F)\leq m_j, \ \ j=0,1,\ldots, k.
\end{equation*}
For this purpose, we use the constraints in \eqref{e:linear.prog}.
Consider first $s_0(F)$. Suppose first that $s_0(G)<m_0$;  then, 
\begin{align*}
  s_0(F) = \sum_{v\in G_0}\lceil ce^{x_v^*}\rceil
  \leq  \sum_{v\in G_0} \bigl( ce^{x_v^*}+1\bigr)
  \leq s_0(G) (cm_0+1) \leq m_0,
\end{align*}
if we choose $c$ to satisfy
$$
c\leq \frac{1}{s_0(G)(1+s_0(G))}.
$$
Suppose next that $s_0(G)=m_0$. In this case, $e^{x_v^*}\leq s_0(G)$ for
each $v\in G_0$, so choosing $c\leq 1/s_0(G)$ leads to
$$
s_0(F)= \sum_{v\in G_0}\lceil ce^{x_v^*} \rceil = s_0(G)=m_0.
$$
Similarly, for any $j=1,\ldots, k$, if $s_j(G)<m_j$, then 
\begin{align*}
s_j(F) = &\sum_{(v_1,\ldots, v_{j+1})\in G_j} \prod_{i=1}^{j+1} \lceil
  ce^{x_{v_i}^*}\rceil
  \leq \sum_{(v_1,\ldots, v_{j+1})\in G_j} \prod_{i=1}^{j+1}\bigl(
           ce^{x_{v_i}^*}+1\bigr) \\
 = &\sum_{(v_1,\ldots, v_{j+1})\in G_j}
     \bigg(1+\sum_{i=1}^{j+1}\sum_{A\subset\{1,\ldots, j+1\}, \#(A)=i}
     c^i \exp\Big\{\sum_{\ell\in A}x_{v_\ell^*}\Big\}\bigg) \\
  \leq &\sum_{(v_1,\ldots, v_{j+1})\in G_j} \bigg(1+\sum_{i=1}^{j+1}
         c^i{j+1\choose i} m_j\bigg) \\
  = &s_j(G) \Big( 1+ m_j\bigl( (1+c)^{j+1}-1\bigr)\Big). 
\end{align*}
Thus, if we choose $c$ to satisfy
$$
c\leq \left[\left(
    \frac{1}{s_j(G)(1+s_j(G))}+1\right)^{1/(j+1)}-1\right], 
$$
it then holds that $s_j(F)\le m_j$, as required. 
On the other hand, if $s_j(G)=m_j$, then for $(v_1,\dots,v_{j+1})\in G_j$, we have $e^{x_{v_i}^*} \le s_j(G)$, $i=1,\dots,j+1$; so choosing $c\le 1/s_j(G)$ again leads to  
$$
s_j(F)=\sum_{(v_1,\dots,v_{j+1})\in G_j} \prod_{i=1}^{j+1}\lceil
ce^{x_{v_i}^*} \rceil = s_j(G)=m_j. 
$$

For the simplicial complex $F$ 
constructed above, the number of ordered copies of $G$ in $F$ is at least
$$
\prod_{v\in G_0}n_v \geq c^{s_0(G)}\exp\Big\{ \sum_{v\in G_0}
  x_v^*\Big\}
= c^{s_0(G)} e^\gamma. 
$$
We thus conclude that
$$
  N(m_0,m_1,\ldots, m_k;G) \geq  N(F,G)\ge   \frac{c^{s_0(G)}}{\#({\rm
    Aut}(G))}e^\gamma,
$$
establishing the lower bound in \eqref{e:both.b.N}.

In order to prove the upper bound in \eqref{e:both.b.N}, we start with 
the dual problem to the optimization problem
\eqref{e:linear.prog}. It is the linear program
\begin{align} \label{e:dual.prog}
&\min\left[\sum_{v\in G_0} y_v  \log m_0+ \sum_{i=1}^k\sum_{\sigma_i\in G_i} z^{(i)}_{\sigma_i} \log
  m_i\right] \\
 &\text{subject to} \notag \\ 
&y_v+\sum_{i=1}^k\sum_{\sigma_i\in G_i \atop v\in
    \sigma_i}z^{(i)}_{\sigma_i} \geq 1 \ \ \text{for any $v\in
                                 G_0$,} \notag \\ 
 & y_v\geq 0, \  z^{(i)}_{\sigma_i} \geq 0 \ \ \text{for all $v\in
   G_0$ and $\sigma_i\in G_i, \ i=1,\ldots,
   k$.} \notag 
\end{align}
The optimal value of the dual problem \eqref{e:dual.prog} equals 
$\gamma=\gamma(m_0,m_1,\ldots, m_k;G)$; that is, it has the same optimal value as the original linear program in \eqref{e:linear.prog}. For later use, let $(y_v^*), \bigl(z_{\sigma_i}^{(i*)}\bigr)$ be an optimal solution to the dual problem in \eqref{e:dual.prog}, so that 
\begin{equation}  \label{e:gamma.yv*.zsigma*}
\gamma=\sum_{v\in G_0} y_v^* \log m_0 + \sum_{i=1}^k \sum_{\sigma_i\in G_i} z_{\sigma_i}^{(i*)} \log m_i. 
\end{equation}

Now, let us fix a simplicial complex $F$ of dimension $k$, satisfying $s_i(F)\leq m_i, \,
i=0,1,\ldots, k$. 
Then, the upper bound in \eqref{e:both.b.N} is obtained as an
immediate consequence of the bound 
\begin{equation}  \label{e:upper.bound.n0FG}
n_o(F,G) \le C(G)e^\gamma
\end{equation}
for some constant $C(G)$ that does not depend on $F$. 
For the proof of \eqref{e:upper.bound.n0FG}, we consider a partition 
\begin{equation} \label{e:partition.F0}
  F_0=\bigcup_{v\in G_0} V_v
\end{equation}
of the vertex set of $F$ into subsets indexed by the vertices of
$G$. 
Denote by $\calH = \calH(F,G)$ the collection of all ordered copies of $G$ in $F$, so that $\#(\calH(F,G))=n_o(F,G)$. Further, let $\calW=\calW(F,G)$ be a subset of $\calH$, such that each $v\in G_0$ is mapped into one of the vertices in $V_v$. 

We create  a  random 
partition \eqref{e:partition.F0} as follows. To each vertex $w\in
F_0$, assign randomly and 
independently a vertex $U(w)\in G_0$. Now, let
$$
V_v=\bigl\{ w\in F_0:\, U(w)=v\bigr\}, \ v\in G_0
$$
(some sets $V_v$ may be empty). In this setting, $\calW$ is  a \emph{random} subset of $\calH$, so that 
\begin{align*}
 & \E\bigl[ \#(\calW)\bigr] =\sum_{\varphi\in \calH}\P(\varphi \in \calW)= \sum_{\varphi\in \calH}\P\Big( U\big(\varphi(v)\big)=v \ \ \text{for each $v\in
    G_0$}\Big) = s_0(G)^{-s_0(G)}n_o(F,G). 
\end{align*}
This indicates that there exists a \emph{nonrandom} partition \eqref{e:partition.F0} of the
vertex set of $F$, for which 
\begin{equation} \label{e:W.not.small}
  \#(\calW) \geq  s_0(G)^{-s_0(G)} n_0(F,G).
 \end{equation}
Fixing a   collection $\calW$ that satisfies \eqref{e:W.not.small},  we define  $\calW_0 := \bigl\{ H_0:\, H\in\calW\bigr\}$ to be
 the collection of {\it ordered} vertex sets of complexes in $\calW$,
 so that $\#(\calW_0)=\#(\calW)$. For a subset $U\subseteq
 F_0$, define the trace of $\calW_0$ on $U$ by
 \begin{equation*}
{\rm Tr}(\calW_0, U)= \bigl\{ H_0\cap U:\, H_0\in \calW_0\bigr\}. 
 \end{equation*}
 We choose a large
 positive integer $t$ and define
$$
l_0(v)=\lceil t y_v^*\rceil, \, v\in G_0, \ \ l_i(\sigma_i)= \lceil t
z_{\sigma_i}^{(i*)}\rceil, \, \sigma_i\in G_i, \, i=1,\ldots, k, 
$$
where $y_v^*$ and $z_{\sigma_i}^{(i*)}$ are given in \eqref{e:gamma.yv*.zsigma*}. 
Referring to the partition \eqref{e:partition.F0} of $F_0$ constructed above
that satisfies \eqref{e:W.not.small}, 
we now construct a family $U_1,\ldots, U_s$ of subsets of $F_0$ 
as follows. Take each set $V_v$ in \eqref{e:partition.F0} exactly
$l_0(v)$ times for every $v\in G_0$. Next, for each $\sigma_i=\{
v_1,\ldots, v_{i+1}\}\in G_i$, $i=1,\ldots, k$, take the union
$V_{v_1}\cup\ldots\cup V_{v_{i+1}}$ exactly $l_i(\sigma_i)$
times. 
Finally, we enumerate these subsets as $U_1,\dots,U_s$, where 
$$
s=\sum_{v\in G_0}l_0(v) +\sum_{i=1}^k \sum_{\sigma_i\in G_i} l_i(\sigma_i). 
$$
Then, for every $v\in G_0$, each of the vertices in $V_v$ appears exactly $l_0(v)+\sum_{i=1}^k \sum_{\sigma_i\in G_i, \, v\in \sigma_i}l_i(\sigma_i)$ times in the sets $U_1,\dots,U_s$. 
By the constraint of the dual problem 
\eqref{e:dual.prog}, 
\begin{align*}
&l_0(v) + \sum_{i=1}^k \sum_{\sigma_i\in G_i \atop v\in
                 \sigma_i}l_i(\sigma_i) 
\geq t\bigg( y_v^*+\sum_{i=1}^k\sum_{\sigma_i\in G_i \atop v\in
    \sigma_i}z^{(i*)}_{\sigma_i}\bigg)\geq t. 
\end{align*}
This implies that every vertex in $F$ appears at least $t$ times in the sets  $U_1,\ldots, U_s$. 
By Lemma 1.2 in
\cite{friedgut:kahn:1998}, 
\begin{align}
  &\bigl( \#(\calW)\bigr)^t = \bigl( \#(\calW_0)\bigr)^t
    \leq \prod_{m=1}^s \#\bigl({\rm Tr}(\calW_0, U_m)\bigr)\label{e:friedgut.kahn}\\
 =& \prod_{v\in G_0}\left[\#\bigl({\rm Tr}(\calW_0,
    V_v)\bigr)\right]^{l_0(v)}
    \prod_{j=1}^k \prod_{\sigma_j=\{v_1,\ldots, v_{j+1}\}\in G_j}
   \left[ \#\bigl({\rm Tr}(\calW_0, V_{v_1}\cup \ldots \cup  V_{v_{j+1}})\bigr)\right]^{l_j(\sigma_j)}. \notag
\end{align}
By the definition of $\calW$, 
each $H_0\in\calW_0$ has at most one element in $V_v$ for each $v\in G_0$, so
\begin{equation}  \label{e:trace.bound1}
\#\bigl({\rm Tr}(\calW_0,  V_v)\bigr) \leq \#(V_v)\leq s_0(F)\leq m_0.
\end{equation}
Similarly, for each $\sigma_j=\{v_1,\ldots, v_{j+1}\}\in G_j$, the intersection of $H_0\in \calW_0$ and $\bigcup_{i=1}^{j+1}V_{v_i}$ either forms a $j$-simplex in $F$ or  becomes an empty set. Therefore, 
\begin{equation}  \label{e:trace.bound2}
\#\bigl({\rm Tr}(\calW_0, V_{v_1}\cup \ldots\cup V_{v_{j+1}})\bigr)
\leq s_j(F)\leq m_j, \ j=1,\ldots, k.
\end{equation}
Substituting the bounds in \eqref{e:trace.bound1} and \eqref{e:trace.bound2} back into \eqref{e:friedgut.kahn}, we have, as $t\to\infty$, 
\begin{align}
\#(\calW) \leq &\prod_{v\in G_0} m_0^{l_0(v)/t} \prod_{j=1}^k
\prod_{\sigma_j=\{v_1,\ldots, v_{j+1}\}\in G_j} m_j^{l_j(\sigma_j)/t} \label{e:limit.calW}\\
  \to &\prod_{v\in G_0} m_0^{y_v^*}\prod_{j=1}^k
\prod_{\sigma_j=\{v_1,\ldots, v_{j+1}\}\in G_j}
        m_j^{z^{(j*)}_{\sigma_j}}=e^\gamma, \notag 
\end{align}
where the last equality follows from \eqref{e:gamma.yv*.zsigma*}. 
Combining \eqref{e:W.not.small} and \eqref{e:limit.calW},  we have 
$$
n_0(F,G) \leq  s_0(G)^{s_0(G)}e^\gamma, 
$$
which establishes the bound \eqref{e:upper.bound.n0FG}, as desired. 
\end{proof}

The following lemma is a useful consequence of Proposition
\ref{prop:N.LP}. It is a higher-dimensional version of Lemma 2.1 in 
\cite{janson:oleszkiewicz:rucinski:2004}.
\begin{lemma} \label{l:compare.N}
Let $H$ be a $k$-dimensional  subcomplex of $G$ with $s_k(H)>0$. Then, there
exists a constant $C_H\in (1,\infty)$ such that if $0\leq m_1<m_2\leq m_0^{k+1}/s_k(G)$,
\begin{align*}
  &N\bigl( m_0, m_1s_1(G), \ldots, m_1s_k(G); H\bigr) 
  \leq C_H \Big(
\frac{m_1}{m_2}\Big)^{1/(k+1)}
N\bigl( m_0, m_2s_1(G), \ldots, m_2s_k(G); H\bigr).
\end{align*}
\end{lemma}
\begin{proof}
  It is enough to consider the case $m_1>0$. 
  By Proposition \ref{prop:N.LP}, 
\begin{align}
  &N\bigl( m_0, m_1s_1(G), \ldots, m_1s_k(G); H\bigr) 
  \leq C_1(H)e^{\gamma_1}  \label{e:Nm1.upper.bound}
\end{align}
for some $C_1(H)\in (1,\infty)$, where
\begin{align*}
  &\gamma_1=\max \sum_{v\in H_0}x_v \\
 &\text{subject to} \notag \\ 
  & 0\leq x_v\leq \log m_0, \ v\in H_0, \notag \\
  &\sum_{v\in \sigma_i}x_v\leq \log s_i(G)+\log m_1, \ 
    \sigma_i\in H_i,\, i=1,\ldots, k.\notag 
\end{align*}
Let $(x_v^*, \, v \in H_0)$ be an optimal solution for this problem. Consider all
$v\in H_0$ that belong to a $k$-simplex in $H$, and choose
among them a vertex 
$\tilde v\in H_0$ with the smallest value of $x_v^*$; that is, 
$$
x_{\tilde v}^* = \min_{\sigma_k \in H_k, \, v\in \sigma_k} x_v^*. 
$$
By the feasibility
of $(x_v^*)$, we have
\begin{equation}  \label{e:bdd.x.v.tilde}
x_{\tilde v}^*\leq \frac{1}{k+1} \log s_k(G) +  \frac{1}{k+1} \log
m_1.
\end{equation}
Define for $v\in H_0$, 
$$
x_v^{**} = x_v^* \ \text{for $v\not= \tilde v$}, \ \ x_{\tilde v}^{**}
= x_{\tilde v}^* + \frac{1}{k+1} \log \frac{m_2}{m_1}.
$$
Then for $ \sigma_i\in H_i, \, i=1,\ldots, k$, we have
\begin{align*}
\sum_{v\in \sigma_i}x_v^{**} &= \sum_{v\in \sigma_i}x_v^{*} +
                 \frac{1}{k+1} \log \frac{m_2}{m_1} \\
 &\leq \log s_i(G) + \log m_1 + \frac{1}{k+1} \log \frac{m_2}{m_1}
  \leq \log s_i(G) + \log m_2, 
\end{align*}
while, by \eqref{e:bdd.x.v.tilde}, 
\begin{align*}
x_{\tilde v}^{**} \leq &\frac{1}{k+1} \log s_k(G) +  \frac{1}{k+1} \log
                         m_1 +  \frac{1}{k+1} \log \frac{m_2}{m_1} \\
 = &\frac{1}{k+1} \log s_k(G) +  \frac{1}{k+1} \log m_2\leq \log m_0.
\end{align*}
We thus conclude that $(x_v^{**})$ is a feasible solution to the linear program
\begin{align*} 
  &\gamma_2=\max \sum_{v\in H_0}x_v \\
 &\text{subject to} \notag \\ 
  & 0\leq x_v\leq \log m_0, \ v\in H_0, \notag \\
  &\sum_{v\in \sigma_i}x_v\leq \log s_i(G)+\log m_2, \  \sigma_i\in
    H_i, \, i=1,\ldots, k. \notag 
\end{align*}
Therefore, 
$$
\gamma_2\geq \sum_{v\in H_0}x_v^{**} = \gamma_1+  \frac{1}{k+1} \log \frac{m_2}{m_1} .
$$
Appealing once again to  Proposition \ref{prop:N.LP}, as well as \eqref{e:Nm1.upper.bound},  we have for some constant
$C_2(H)\in (0,1)$, 
\begin{align*}
&N\bigl( m_0, m_2s_1(G), \ldots, m_2s_k(G); H\bigr) \geq C_2(H)
  e^{\gamma_2} \geq C_2(H) \left( \frac{m_2}{m_1} \right)^{1/(k+1)}
  e^{\gamma_1}  \\
  &\quad \ge \frac{C_2(H)}{C_1(H)}  \left( \frac{m_2}{m_1}
  \right)^{1/(k+1)} \hspace{-5pt} N\bigl( m_0, m_1s_1(G), \ldots, m_1s_k(G); H\bigr), 
\end{align*}
as required. 
\end{proof}

For numbers $0\leq p_i\leq 1, \, i=1,\ldots, k$, and a simplicial complex $G$ of dimension $k$, denote 
 \begin{equation*}
 \Psi_{G,n} :=n^{s_0(G)}\prod_{i=1}^k
p_i^{s_i(G)}  \sim \mu_{o,n}(G), \ n\to\infty,
\end{equation*}
and define
\begin{align*}
M^*_{G,n}(p_1,\ldots, p_k) := \max\biggl\{ &1\leq m\leq \frac{{n\choose k+1}}{s_k(G)}:
        N\bigl( n, ms_1(G), \ldots, ms_k(G); H\bigr) \leq \Psi_{H,n}  \\
 & \qquad \qquad  \qquad  \text{for every non-empty subcomplex $H$ of
   $G$}\biggr\}.     \notag 
\end{align*}

The following theorem is the main result of this section.  It is an
extension of Theorem 1.2 of \cite{janson:oleszkiewicz:rucinski:2004}
to the multi-parameter random simplicial complexes. 
Its statement uses
the notation $K_{k,n}$ for the complete complex of dimension $k$ of
$n$ vertices (i.e. a complex on $n$ vertices containing all possible
simplices of dimensions $k$ and smaller). 

\begin{theorem} \label{t:upper}
For every $\vep>0$, 
there exists $C(\vep,G)>0$ so that for all $n\ge1$, 
\begin{equation} \label{e:tail.prob.upper1.0}
  \P\bigl(N_{o,n}(G)\ge (1+\vep)\mu_{o,n}(G)\bigr) \leq \exp\Bigl\{ -C(\vep,G)
  M^*_{G,n}(p_1,\ldots,p_k)\Bigr\}. 
\end{equation}
Moreover, if $(1+\vep)\mu_{o,n}(G)\leq N\bigl( K_{k,n},G\bigr)$, there exists 
$B(\vep,G)>0$, such that for all $n\ge 2k+1$, 
\begin{align}  \label{e:lower.b4.0}
  &\P\bigl( N_{o,n}(G)\geq (1+\vep)\mu_{o,n}(G)\bigr) 
   \geq     \frac14 \left(  \prod_{j=1}^k
    p_j\right)^{B(\vep,G)M^*_{G,n}(p_1,\ldots,p_k)}.
\end{align}  
\end{theorem}

\begin{remark} \label{rk:main}
Theorem \ref{t:upper} identifies the order of magnitude of the upper tail large
deviation probability  
at the  logarithmic scale precision. 
The logarithmic order of magnitude  
differs between the upper bound \eqref{e:tail.prob.upper1.0} and the
lower bound \eqref{e:lower.b4.0} by a factor of $\log \big(\prod_{j=1}^k p_j\big)$. 
    We note that under a common setup $p_i= n^{-\alpha_i}$ for some $\alpha_i\in [0,\infty]$, $i\ge1$, as in Section \ref{sec:simplices} below, the  factor $\log \big(\prod_{j=1}^k p_j\big)$  is logarithmic in $n$.  In contrast,  
    the main term $M^*_{G,n}(p_1,\ldots, p_k)$ typically grows
    polynomially, determining largely the order of magnitude of the
    large deviation probability.


Note also that the condition $(1+\vep)\mu_{o,n}(G)\leq N\bigl(
K_{k,n},G\bigr)$ is rarely restrictive. In fact, it is equivalent to
$(1+\vep) \prod_{j=1}^k p_j^{s_j(G)} \le1$; this, however, trivially
holds whenever $p_j\to0$ as $n\to\infty$ for some $j$. 


\end{remark}

\begin{proof}
We start with proving the upper bound in \eqref{e:tail.prob.upper1.0}. Let
$G_1,\ldots, G_M$ be  the ordered copies of $G$ in $K_{k,n}$ where $M=(n)_{s_0(G)}=n(n-1)\cdots \bigl( n-s_0(G)+1\bigr)$. Clearly, 
\begin{equation*}
  N_{o,n}(G)=\sum_{j=1}^M I_j,
\end{equation*}
where
$$
I_j = \one\bigl\{ G_j  \text{ is a subcomplex of }  K(n; p_1,\ldots,
p_k)\bigr\}, \ \ j=1,\ldots, M.
$$
Therefore, for each $m=1,2\ldots$, we have
\begin{align} \label{e:moment.m0}
\E [N_{o,n}(G)^m] = \sum_{1\le i_1,\ldots,i_m\le M} \E\bigl[ I_{i_1}\cdots I_{i_m}\bigr]
= \sum_{1\le i_1,\ldots,i_m\le M} \prod_{j=1}^k p_j^{s_j( G_{i_1}\cup\cdots\cup
  G_{i_m})}, 
\end{align}
with $s_j(\cdot)$ as in \eqref{e:number.simp}. For a fixed  ${\bf i}^{(m-1)}=(i_1,\ldots, 
i_{m-1})\in \{ 1,\dots,M \}^{m-1}$, denote $F_{{\bf
    i}^{(m-1)}}=G_{i_1}\cup\cdots\cup   G_{i_{m-1}}$. Then
\eqref{e:moment.m0} becomes 
\begin{align*}
\E [N_{o,n}(G)^m] = \sum_{{\bf i}^{(m-1)}} \prod_{j=1}^k p_j^{s_j(F_{{\bf
    i}^{(m-1)}})} \sum_{i_m=1}^M \prod_{j=1}^k p_j^{s_j(G)-s_j(F_{{\bf
    i}^{(m-1)}}\cap G_{i_m})}. 
\end{align*}
For every fixed ${\bf  i}^{(m-1)}$, 
if $F_{{\bf i}^{(m-1)}}\cap G_{i_m}$ contains at least one simplex of positive
dimension,  then this  intersection is isomorphic to some subcomplex $H$ of $G$ of positive dimension. 
 We thus conclude that 
\begin{align}
\E [N_{o,n}(G)^m] &\le \sum_{{\bf i}^{(m-1)}} \prod_{j=1}^k p_j^{s_j(F_{{\bf i}^{(m-1)}})}   \bigg[ M \prod_{j=1}^k p_j^{s_j(G)} + \hspace{-10pt}\sum_{i_m: F_{{\bf i}^{(m-1)}} \cap G_{i_m} \neq\emptyset}  \prod_{j=1}^k p_j^{s_j(G)-s_j(F_{{\bf i}^{(m-1)}}\cap G_{i_m})} \bigg] \notag \\
&=\sum_{{\bf i}^{(m-1)}} \prod_{j=1}^k p_j^{s_j(F_{{\bf i}^{(m-1)}})} \bigg[ \mu_{o,n}(G) + \sum_{H\subseteq G}  \prod_{j=1}^k p_j^{s_j(G)-s_j(H)} \#\big\{i: F_{{\bf i}^{(m-1)}} \cap G_i \cong H\big\} \bigg], \notag 
\end{align}
where $\mu_{o,n}(G)$ is given in \eqref{e:mean.ordered}, and the sum $\sum_{H\subseteq G}$ is taken over all subcomplexes of $G$ with positive dimension, and $\cong$ means isomorphism between simplicial complexes. 
For every subcomplex $H$ of $G$, there are at most
\begin{align*}
&N\bigl( s_0(F_{{\bf i}^{(m-1)}}), s_1(F_{{\bf i}^{(m-1)}}), \ldots,
                 s_k(F_{{\bf  i}^{(m-1)}}); H\bigr) 
\leq N\bigl( n, (m-1)s_1(G),\ldots, (m-1)s_k(G); H\bigr)
\end{align*}
ways to choose an unordered copy of $H$ in $F_{{\bf i}^{(m-1)}}$. To
bound the number of 
ordered copies of $G$ in $K_{k,n}$ whose intersection with $F_{{\bf
    i}^{(m-1)}}$ is isomorphic to $H$  (i.e., $\#\{i: F_{{\bf i}^{(m-1)}} \cap G_i \cong H\}$), notice that each choice of
an unordered copy of $H$ in $F_{{\bf i}^{(m-1)}}$ determines $s_0(H)$ vertices in the
copy of $G$; thus,  the number of ways to select the remaining vertices
of the copy of $G$ is at most $\big( n-s_0(H) \big)_{s_0(G)-s_0(H)} = (n)_{s_0(G)}/(n)_{s_0(H)}$. 
Finally, the vertices of the copy of $G$ can be numbered in at most
$s_0(G)!$ ways. From these observations, we conclude that 
\begin{align*}
&\#\big\{i: F_{{\bf i}^{(m-1)}} \cap G_i \cong H\big\}  \le N\bigl( n, (m-1)s_1(G),\ldots, (m-1)s_k(G); H\bigr) 
\frac{(n)_{s_0(G)}}{(n)_{s_0(H)}} \, s_0(G)!. 
\end{align*}
Now \eqref{e:mean.ordered} gives us
\begin{align*} 
\E [N_{o,n}(G)^m] \leq & \sum_{{\bf i}^{(m-1)}} \prod_{j=1}^k p_j^{s_j(F_{{\bf
  i}^{(m-1)}})} \\
  & \times \mu_{o,n}(G)\bigg[ 1 + s_0(G)!\sum_{H \subseteq G}
  \frac{N\bigl( n, (m-1)s_1(G),\ldots, (m-1)s_k(G); H\bigr)}{\mu_{o,n}(H)}
  \bigg].  
\end{align*}
Using \eqref{e:moment.m0}  with $m$ replaced by $m-1$ results in 
$$
\E [N_{o,n}(G)^m] \leq \E[N_{o,n}(G)^{m-1}]\mu_{o,n}(G)\bigg[ 1+ s_0(G)!\sum_{H \subseteq G}
\frac{N\bigl( n, (m-1)s_1(G),\ldots, (m-1)s_k(G);
  H\bigr)}{\mu_{o,n}(H)}\bigg]. 
$$
By the monotonicity of the function $N$ in all of its arguments, 
an inductive argument gives us the bound
\begin{equation}  \label{e:moment.m.f}
  \E [N_{o,n}(G)^m] \leq \mu_{o,n}(G)^{m} \Biggl[ 1+ 
   s_0(G)!\sum_{H \subseteq G}
\frac{N\bigl( n, (m-1)s_1(G),\ldots, (m-1)s_k(G);
  H\bigr)}{\mu_{o,n}(H)}\Biggr]^{m-1}, 
\end{equation}
for every $m\ge1$.

For $\theta\in (0,1)$ to be determined in the sequel, we take
$m=\lceil\theta M^*_{G,n}(p_1,\ldots,p_k)\rceil=:\lceil\theta M_{G,n}^*\rceil$. Note that by
Lemma \ref{l:compare.N},
\begin{align*}
  &N\bigl( n, (m-1)s_1(G),\ldots, (m-1)s_k(G);  H\bigr) \\
  \leq &N\bigl( n, \theta M_{G,n}^* s_1(G),\ldots, \theta M_{G,n}^* s_k(G);
  H\bigr) \\
\leq &  C_H \theta^{1/(k+1)} N\bigl( n,  M_{G,n}^* s_1(G),\ldots,  M_{G,n}^* s_k(G);
  H\bigr) \leq  \theta^{1/(k+1)}  C_H \Psi_{H,n}. 
\end{align*}
Therefore, 
using \eqref{e:moment.m.f} with $m=\lceil\theta M_{G,n}^*\rceil$ and Markov's inequality, 
we obtain for $\vep> 0$,
\begin{align*}
  &\P\bigl(N_{o,n}(G)\ge (1+\vep)\mu_{o,n}(G)\bigr) \\
  &\le (1+\vep)^{-\theta M_{G,n}^*} \bigg[
  1+   \theta^{1/(k+1)} s_0(G)!
      \sum_{H 
  \subseteq G} \frac{C_H\Psi_{H,n}}{\mu_{o,n}(H)}\bigg]^{\theta M_{G,n}^*} \notag \\
&=   (1+\vep)^{-\theta M_{G,n}^*} \bigg[ 1+ \theta^{1/(k+1)}  s_0(G)!\sum_{H
  \subseteq G} \frac{C_H n^{s_0(H)}}{n(n-1)\cdots \bigl(
   n-s_0(H)+1\bigr)}\bigg]^{\theta M_{G,n}^*}  \notag \\
 &\le   \bigg[ (1+\vep)^{-1} \Big( 1+ \theta^{1/(k+1)} s_0(G)^{s_0(G)} \sum_{H \subseteq G} C_H\Big)\bigg]^{\theta M_{G,n}^*}. \notag  
\end{align*}
Choosing 
$$
0<\theta<\min\bigg\{   1, \bigg(\frac{\vep}{ s_0(G)^{s_0(G)}
    \sum_{H \subseteq G} C_H}  \bigg)^{k+1}  \bigg\}
$$
establishes \eqref{e:tail.prob.upper1.0}.

Our proof of \eqref{e:lower.b4.0} is also inspired by an argument in
\cite{janson:oleszkiewicz:rucinski:2004}.   Suppose there exist $a_i>0, \, i=1,\ldots,
k$, $m\ge1$, and a subcomplex $H$ of $G$, such that 
\begin{equation} \label{e:cond.lower}
N\bigl( n,a_1m,\ldots, a_km; H\bigr)\geq 2(1+\vep)\Psi_{H,n}.
\end{equation} 
This implies that there is a complex $F$ on at most $n$ nodes with 
$s_i(F)\leq a_im, \, i=1,\ldots, k$, such that
\begin{equation} \label{e:lower.N}
n(F,H)\geq 2(1+\vep)\Psi_{H,n}\geq 2(1+\vep) \mu_{o,n}(H).
\end{equation}
Our first goal is to show that under the assumption \eqref{e:cond.lower}, 
\begin{equation}   \label{e:lower.b1}
\P\big( N_{o,n}(G)\geq (1+\vep)\mu_{o,n}(G)\big)\ge \frac14 \prod_{j=1}^k  p_j^{s_j(G)+a_jm}. 
\end{equation}

Given a subcomplex  $H$ of $G$ satisfying \eqref{e:cond.lower}, we see that each  of the ordered copies of $G$ in
$K_{k,n}$ has a 
unique corresponding ordered copy of $H$ in $K_{k,n}$; if the latter
is also in $F$, we refer to  that ordered copy of $G$ as being $F$-rooted. Since there
are $n_o(F,H) = \#({\rm Aut}(H)) n(F,H)$ ordered copies of $H$ in $F$, the number of
  $F$-rooted ordered copies of $G$ in $K_{k,n}$ is
  $$
  J := n_o(F,H)\bigl(n-s_0(H)\bigr)_{s_0(G)-s_0(H)}. 
  $$
Denote these ordered copies of $G$ in $K_{k,n}$ by $G_1,\ldots, G_J$. Since $\#({\rm
   Aut}(H))\geq 1$, it follows from \eqref{e:lower.N} that 
 \begin{align} \label{e:mean.not.l}
   J\prod_{j=1}^k p_j^{s_j(G)-s_j(H)}\geq
   2(1+\vep)\mu_{o,n}(H)\bigl(n-s_0(H)\bigr)_{s_0(G)-s_0(H)}\prod_{j=1}^k
   p_j^{s_j(G)-s_j(H)} =2(1+\vep)\mu_{o,n}(G).
\end{align}

Let  $K_F(n; p_1,\ldots, p_k)$ be the multi-parameter simplicial complex $K(n; p_1,\ldots,
p_k)$ conditioned on $F\subseteq K(n; p_1,\ldots, p_k)$. For
$i=1,\ldots, J$, let $Z_i$ be the indicator function of the event that $G_i$ is a subcomplex of 
$K_F(n; p_1,\ldots, p_k)$. Then, 
\begin{equation}  \label{e:lower.bdd.indep.of.i}
  \P(Z_i=1) = \prod_{j=1}^k p_j^{s_j(G_i\setminus F)}
  \geq \prod_{j=1}^k p_j^{s_j(G)-s_j(H)}.
\end{equation}
Since the rightmost term in \eqref{e:lower.bdd.indep.of.i} is independent of $i$, the lower bound in Lemma 3.3 of
\cite{janson:oleszkiewicz:rucinski:2004} gives us 
\begin{align*}
&\P\bigg( N_{o,n}(G)\geq \frac{J}{2} \prod_{j=1}^k p_j^{s_j(G)-s_j(H)}\Bigg|
  F\subseteq K(n; p_1,\ldots, p_k)\bigg) \\
\geq &\P\bigg( \sum_{i=1}^J Z_i\geq \frac{J}{2} \prod_{j=1}^k
       p_j^{s_j(G)-s_j(H)} \bigg) 
\geq \frac14 \prod_{j=1}^k  p_j^{s_j(G)-s_j(H)} \geq \frac14 \prod_{j=1}^k  p_j^{s_j(G)}.
\end{align*}
Therefore, by \eqref{e:mean.not.l}, 
\begin{align*}
&\P\big( N_{o,n}(G)\geq (1+\vep)\mu_{o,n}(G)\big)\geq \P\Big( N_{o,n}(G)\geq \frac{J}{2}
                \prod_{j=1}^k p_j^{s_j(G)-s_j(H)}\Big) \\
\geq & \frac14 \prod_{j=1}^k  p_j^{s_j(G)}\P\big( F\subseteq K(n;
       p_1,\ldots, p_k)\big)   \\
  =& \frac14 \prod_{j=1}^k  p_j^{s_j(G)} \prod_{j=1}^k  p_j^{s_j(F)}
  \geq  \frac14 \prod_{j=1}^k  p_j^{s_j(G)+a_jm}, 
\end{align*}
establishing \eqref{e:lower.b1}, as desired. 

Now, we are ready to prove the lower bound in \eqref{e:lower.b4.0}.
Suppose first that
\begin{equation} \label{e:Mg.small}
  2\big(2(1+\vep)C_H\big)^{k+1}M_{G,n}^*\leq s_k(G)^{-1}  {n\choose k+1}, 
\end{equation}
where $C_H>1$ is the constant in Lemma \ref{l:compare.N}, that is increased, without loss of generality, to
be the nearest positive integer. Then,  $M_{G,n}^*<\lfloor s_k(G)^{-1}{n\choose k+1}\rfloor$, so there is
a subcomplex $H$ of $G$ such that
$$
N\Bigl( n, \bigl( M_{G,n}^* +1\bigr) s_1(G), \ldots,  \bigl( M_{G,n}^* +1\bigr)
s_k(G); H\Bigr) >\Psi_{H,n}.
$$
Therefore,  by Lemma \ref{l:compare.N}, 
\begin{align*}
2(1+\vep)\Psi_{H,n}<&2(1+\vep) N\Bigl( n, \bigl( M_{G,n}^* +1\bigr) s_1(G), \ldots,  \bigl( M_{G,n}^* +1\bigr)
s_k(G); H\Bigr) \\
  \leq &2(1+\vep) N\Bigl( n, 2M_{G,n}^* s_1(G), \ldots,  2M_{G,n}^* s_k(G); H\Bigr) \\
  \leq & N\Bigl( n,2(2(1+\vep)C_H)^{k+1}M_{G,n}^*s_1(G), \ldots,
         2(2(1+\vep)C_H)^{k+1}M_{G,n}^*s_k(G); H\Bigr). 
\end{align*}
Since the condition \eqref{e:cond.lower} is now satisfied with $a_i=s_i(G)$, $i=1,\dots,k$, and $m=2\big(2(1+\vep) C_H \big)^{k+1}M_{G,n}^*$,  we conclude by
\eqref{e:lower.b1} that
\begin{align} \label{e:lower.b2}
  &\P\big( N_{o,n}(G)\geq (1+\vep)\mu_{o,n}(G)\big)\geq  
    \frac14 \prod_{j=1}^k  p_j^{s_j(G)+2(2(1+\vep)C_H)^{k+1}M_{G,n}^*s_j(G) }
   \geq      \frac14 \bigg(  \prod_{j=1}^k  p_j\bigg)^{B_1(\vep,G)M_{G,n}^*},
\end{align}
where
$$
B_1(\vep,G)=\max_{j=1,\ldots,k}s_j(G)\Bigl[ 1+ 2\bigl( 2(1+\vep)\max_{H\subseteq
  G} C_H\bigr)^{k+1}\Bigr].
$$

Next, we need to consider the case when \eqref{e:Mg.small}  does not hold. 
In this case, it follows from the assumption 
$(1+\vep)\mu_{o,n}(G)\leq N(K_{k,n},G)$ that, for $n\geq 2k+1$, 
\begin{align} \label{e:lower.b3}
  &\P\big( N_{o,n}(G)\geq (1+\vep)\mu_{o,n}(G)\big)\geq  \P\big( N_{o,n}(G)\geq  N\bigl(
    K_{k,n},G\bigr)\big) \\
  \geq &\P\Bigl( K(n; p_1,\ldots, p_k) = K_{k,n}\Bigr)
         = \prod_{j=1}^k  p_j^{n\choose j+1}
  \geq \prod_{j=1}^k  p_j^{n\choose k+1}
    \geq \bigg( \prod_{j=1}^k  p_j\bigg)^{B_2(\vep,G)M_{G,n}^*}, \notag
\end{align}
where
$$
B_2(\vep,G)=2\bigl( 2(1+\vep)\max_{H\subseteq 
  G} C_H\bigr)^{k+1} s_k(G). 
$$
Now, \eqref{e:lower.b4.0} follows from \eqref{e:lower.b2} and
\eqref{e:lower.b3} . 
\end{proof}

\section{Simplices at the critical dimension and below} 
\label{sec:simplices}

Distributional limit theorems for the multi-parameter simplicial
complex $K(n; \bp)$ were obtained in
\cite{owada:samorodnitsky:thoppe:2021}. These 
results are obtained under the assumption 
\begin{equation} \label{e:power.p}
  p_i=n^{-\alpha_i}, \ i\ge1,
\end{equation}
for $\alpha_i \in [0,\infty]$, $i\ge1$. In this section we retain this
assumption and, instead of distributional results, we investigate
large deviation probabilities for the number of certain simplices in
$K(n; \bp)$. We will use the general results obtained in
the previous section.

We are particularly interested in counting the simplices at or below the
\emph{critical dimension} of the model. In other words, we consider the simplices of dimension $k\ge1$, satisfying 
\begin{equation} \label{e:subcritical.k}
  \sum_{i=1}^k {k\choose i} \alpha_i<1, \ \ \text{ and } \ \ q:=\min\{ i\ge1: \alpha_i>0 \} \le k; 
\end{equation}
this introduces a minor unimportant abuse of notation by conflating the dimension
of the simplex with the largest dimension of an entire complex. Counting
the simplices
at the critical dimension is particularly important since their
numbers largely determine the behaviour of the Euler characteristic
of  $K(n; \bp)$; 
\cite{thoppe:yogeshwaran:adler:2016, owada:samorodnitsky:thoppe:2021}.

Let $\sigma_k$ be a simplex of dimension $k$ satisfying \eqref{e:subcritical.k}.
The following proposition, a part of which requires an extra
assumption on the parameters, describes the size of a 
crucial ingredient in the logarithmic order of magnitude of the upper tail
large deviations probability: $\P\bigl(N_{o,n}(\sigma_k)>(1+\vep)
\mu_{o,n}(\sigma_k)\bigr)$ for $\vep>0$. Notice that the extra assumption \eqref{e:new.cond} below will be used only for proving a lower bound in \eqref{e:conj.M}. 
\begin{proposition} \label{prop:simplex}
Suppose  that \eqref{e:subcritical.k} holds. In the case of $q<k$, suppose also that for any $k_0=q+1,\ldots, k$, 
\begin{equation} \label{e:new.cond}
\frac{k-q}{k+1}\binom{k+1}{q+1}\alpha_q+ \sum_{j=q+1}^{k_0} {k+1 \choose j+1}\alpha_j
 <k_0-q. 
\end{equation}
Then, 
for large enough $n$, 
\begin{equation} \label{e:conj.M}
 C_k^{-1}n^{q+1-\binom{k}{q}\alpha_q}\leq  M_{\sigma_k,n}^*\bigl(n^{-\alpha_1},\ldots,
 n^{-\alpha_k}\bigr)\leq C_kn^{q+1-\binom{k}{q}\alpha_q}
 \end{equation}
 for some $C_k\geq 1$. In particular, for $\vep>0$ and  all large
 enough $n$, 
 \begin{equation} \label{e:LDP.simplex}
\exp\Bigl\{ -C_k^\prime(\vep)n^{q+1-\binom{k}{q}\alpha_q}\log n\Bigr\} \leq 
 \P\bigl(N_{o,n}(\sigma_k)\ge (1+\vep)\mu_{o,n}(\sigma_k)\bigr) \leq
 \exp\Bigl\{ -C_k^{\prime\prime}(\vep)n^{q+1-\binom{k}{q}\alpha_q}\Bigr\}
\end{equation}
for some positive constants $C_k^\prime(\vep), \, C_k^{\prime\prime}(\vep)$. 
\end{proposition}

\begin{remark} \label{rk:1skel}
Proposition \ref{prop:simplex} indicates that, at least under an extra
condition, it is the skeleton of dimension $q$, the lowest non-trivial dimension of the complex,  that plays a crucial role in
determining the rate of decay of the upper large deviation
probabilities at the critical dimension and below it. The reason appears to be the fact that ``flipping" of a $q$-simplex from ``on" to ``off" or vice verse affects the topology of the complex more than does any flipping in other dimensions. The same phenomenon has already been observed in the central limit theorem for the simplex counts at the critical dimension; see Proposition 3.6 in \cite{owada:samorodnitsky:thoppe:2021}. 
\end{remark}
\begin{remark}\label{rk:extra.cond} {\rm
Suppose  $q=1$ in \eqref{e:subcritical.k}. Then, if $k=1$, the statement of Proposition \ref{prop:simplex} 
follows from Corollary 1.7 in \cite{janson:oleszkiewicz:rucinski:2004}. For $k=2$, 
it is easy to see that condition \eqref{e:new.cond} follows from the
subcriticality condition \eqref{e:subcritical.k}, but that is no longer the case for 
$k\geq 3$. However, if $k=3$, one can still directly compute the value of $M_{\sigma_3,n}^*(n^{-\alpha_1}, n^{-\alpha_2}, n^{-\alpha_3})$ and verify the inequalities in \eqref{e:conj.M}, without using condition \eqref{e:new.cond}. 
To summarize, the claim of Proposition \ref{prop:simplex} holds at least for $k\in\{1,2,3\}$, under the assumption \eqref{e:subcritical.k} only. We do not know if one can deduce the same conclusion for   $k\ge4$. }
\end{remark}

\begin{proof}
Since \eqref{e:LDP.simplex} follows from \eqref{e:conj.M} and Theorem
\ref{t:upper}, we only need to prove the bounds in \eqref{e:conj.M}. 
We start with the upper bound.  Recall that 
$$
M_{\sigma_k,n}^*\bigl(n^{-\alpha_1},\ldots,
     n^{-\alpha_k}\bigr) = \min_{H:  \text{  subcomplex of $\sigma_k$}}  K_H,
$$
where for a subcomplex $H$ of $\sigma_k$, 
\begin{align} \label{e:K.H}
K_H=\max\biggl\{ m\leq {n\choose k+1}:  
  N\left( n, m{k+1 \choose 2}, m\binom{k+1}{3}, \ldots, m\binom{k+1}{k}, m; H\right) \leq
  \Psi_{H,n}\biggr\}. 
\end{align} 
Therefore, to prove the upper bound in \eqref{e:conj.M} we only need
to detect a specific subcomplex $H$ of $\sigma_k$, such that
\begin{equation} \label{e:conj.upper.proof}
  K_H \leq C_kn^{q+1-\binom{k}{q}\alpha_q}.
\end{equation} 

Let us take $H$ to be the $q$-skeleton of $\sigma_k$. For this $H$, in
the obvious notation, 
\begin{equation} \label{e:Kh.1skel}
K_H=\max\biggl\{ m\leq {n\choose k+1}:  
  N\left( n, m{k+1 \choose 2}, \dots, m\binom{k+1}{q+1}; H\right) \leq   n^{k+1- {k+1 \choose q+1}\alpha_q}\biggr\}. 
\end{equation}
By Proposition \ref{prop:N.LP}, 
\begin{equation} \label{e:kgen.both}
  a_ke^\gamma\leq N\left( n,  m{k+1 \choose 2}, \dots, m\binom{k+1}{q+1}; H\right)  \leq b_ke^\gamma
\end{equation}
for some $a_k,b_k>0$, where 
\begin{align}
  &\gamma= \max \sum_{v=1}^{k+1}x_v  \label{e:one.dim.linear.prog}\\
&\text{subject to}   \notag \\ 
  & 0\leq x_v  \leq \log n, \ v=1, \ldots, k+1, \notag \\
 &\sum_{v\in \sigma_i} x_v\le \log \bigg\{  m\binom{k+1}{i+1} \bigg\}, \ \ \sigma_i\in H_i, \, i=1,\dots,q. \notag 
\end{align}
First, suppose $\binom{k+1}{q+1}m>n^{q+1}$, in which case, we have 
 $\binom{k+1}{j+1}m>n^{j+1}$, $j=1,\dots,q$. Then, $x_v=\log n$, $v=1,\dots,k+1$, is easily seen to be an optimal solution to \eqref{e:one.dim.linear.prog}. It then follows from \eqref{e:kgen.both} that 
\begin{equation}  \label{e:lower.bdd.N.q}
N\left( n,  m{k+1 \choose 2}, \dots, m\binom{k+1}{q+1}; H\right) \ge a_kn^{k+1}. 
\end{equation}
However, as $\alpha_q>0$, there is no $m\in \bbn$ that satisfies \eqref{e:lower.bdd.N.q} and the inequality in \eqref{e:Kh.1skel}. 

Therefore, we only need to consider the case $\binom{k+1}{q+1}m\le
n^{q+1}$. Then, one can see that 
$$
x_v = \frac{1}{q+1}\log \bigg\{  m\binom{k+1}{q+1} \bigg\},  \ \ v=1,\dots,k+1,
$$
is an optimal solution to the linear program \eqref{e:one.dim.linear.prog}, so that 
$$
\gamma = \frac{k+1}{q+1} \log \bigg\{  m\binom{k+1}{q+1} \bigg\}. 
$$
Therefore, by \eqref{e:kgen.both}, 
$$
 b_k^{-\frac{q+1}{k+1}} \binom{k+1}{q+1}^{-1}n^{q+1-\binom{k}{q}\alpha_q} \leq K_H\leq
 a_k^{-\frac{q+1}{k+1}} \binom{k+1}{q+1}^{-1} n^{q+1-\binom{k}{q}\alpha_q}, 
$$
and \eqref{e:conj.upper.proof} follows.

We now prove the lower bound in  \eqref{e:conj.M}. For this purpose we need to prove that for every subcomplex $H$
of $\sigma_k$, 
\begin{equation} \label{e:conj.lower.proof}
  K_H \geq C_k^{-1}n^{q+1-\binom{k}{q}\alpha_q}.
\end{equation} 
Consider first a subcomplex $H$ of dimension $1,\dots,q-1$. In this case, it is clear that $\Psi_{H,n}=n^{k+1}$, and thus, $K_H=\binom{n}{k+1}$ and \eqref{e:conj.lower.proof} trivially holds. 
Consider next a subcomplex $H$ of dimension $q$. Let $\bar H$ be a $(q+1)$-uniform hypergraph on $k+1$ vertices with its hyperedges identified as a $q$-simplex in $H$. Given another hypergraph $\bar F$, define $\bar n(\bar F, \bar H)$ to be the number of unordered copies (as a hypergraph) of $\bar H$ in $\bar F$. Define also 
$$
\bar N \left(n,m\binom{k+1}{q+1}; \bar H  \right) :=\max \bigg\{ \bar n(\bar F, \bar H): v_{\bar F}\le n, \ e_{\bar F}\le m \binom{k+1}{q+1}  \bigg\}, 
$$
where $v_{\bar F}$ is the number of vertices in $\bar F$ and $e_{\bar F}$ the number of hyperedges in $\bar F$. Then, by construction, we have 
$$
\bar N \left(n,m\binom{k+1}{q+1}; \bar H  \right)  \ge N \left(n,m\binom{k+1}{2}, \dots, m\binom{k+1}{q+1};  H  \right). 
$$
By virtue of this inequality together with $v_{\bar H}=k+1$ and $e_{\bar H}=s_q(H)$, 
\begin{align}
K_H &= \max \bigg\{ m\le \binom{n}{k+1}: N \left(n,m\binom{k+1}{2}, \dots, m\binom{k+1}{q+1};  H  \right) \le n^{k+1-s_q(H)\alpha_q}  \bigg\}  \label{e:KH.lower.2nd.case}\\
&\ge  \max \bigg\{ m\le \binom{n}{k+1}:  \bar N \left(n,m\binom{k+1}{q+1}; \bar H  \right) \le n^{v_{\bar H}} p_q^{e_{\bar H}} \bigg\}. \notag 
\end{align}
Observe that $\bar H$ is seen to be a subhypergraph of a $\binom{k}{q}$-regular, $(q+1)$-uniform hypergraph. Moreover, by \eqref{e:subcritical.k}, 
$$
p_q=n^{-\alpha_q} \ge n^{-\binom{k}{q}^{-1}} > n^{-(q+1)\binom{k}{q}^{-1}}. 
$$
Hence, Proposition 4.3 in \cite{dudek:polcyn:rucinski:2010} implies that the last quantity in \eqref{e:KH.lower.2nd.case} is at least 
$$
Cn^{q+1}p_q^{\binom{k}{q}} = C n^{q+1-\binom{q}{k}\alpha_q}, 
$$
for some constant $C$, as desired for \eqref{e:conj.lower.proof}. 


Now, it remains  to establish
 \eqref{e:conj.lower.proof} for subcomplexes $H$ of dimension
 $k_0=q+1,\ldots, k$.
 By \eqref{e:K.H}, we need to show that there exists $C_k>0$ such
 that for any subcomplex $H$ of $\sigma_k$  on
 $k+1$ vertices, and all $n$ large enough,
 \begin{align*}
 &N\left( n, \lceil C_k^{-1} n^{q+1-\binom{k}{q}\alpha_q}\rceil {k+1 \choose 2},
   \ldots, \lceil C_k^{-1} n^{q+1-\binom{k}{q}\alpha_q}\rceil {k+1 \choose k_0+1};
   H\right)     \leq n^{k+1-\sum_{j=q}^{k_0}s_j(H)\alpha_j}.
 \end{align*}
It follows from Proposition \ref{prop:N.LP}, together with the dual formulation \eqref{e:dual.prog}, 
 that for a given subcomplex $H$ and  $n$
large enough, 
it is sufficient to exhibit non-negative numbers $y_v,\, v=1,\ldots,k+1$ and
$z_{\sigma_i}^{(i)}, \, \sigma_i\in H_i, \, i=1,\ldots, k_0$, such that
\begin{equation} \label{e:dual.cond.0}
  y_v+\sum_{i=1}^{k_0}\sum_{\sigma_i\in H_i \atop v\in
    \sigma_i}z^{(i)}_{\sigma_i} \geq 1 \ \ \text{for any $v=1,\ldots,k+1$}, 
\end{equation}
and
\begin{align} \label{e:small.sum}
&\sum_{v=1}^{k+1}y_v \log n+ \sum_{i=1}^{k_0} \sum_{\sigma_i\in H_i}
z^{(i)}_{\sigma_i} \log \left(\lceil C_k^{-1} n^{q+1-\binom{k}{q}\alpha_q}\rceil
  {k+1 \choose i+1}\right) \\
\notag \leq& \Big(   k+1-\sum_{j=q}^{k_0}s_j(H)\alpha_j\Big)\log n +B, 
\end{align}
where $B$ is a $k$-dependent constant. It is 
clear that if we can choose these numbers in such a way that 
\begin{equation} \label{e:small.sum2}
  \sum_{v=1}^{k+1}y_v + \Big(q+1-\binom{k}{q}\alpha_q\Big)\sum_{i=1}^{k_0}
  \sum_{\sigma_i\in H_i}  z^{(i)}_{\sigma_i}   <
  k+1-\sum_{j=q}^{k_0}s_j(H)\alpha_j,
\end{equation}
then \eqref{e:small.sum} will be satisfied for large $n$, regardless of
 the constant $C_k$  above. Specifically, 
we choose the numbers $(y_v)$ and $(z_{\sigma_i}^{(i)})$ as follows. 
\begin{align*} 
&y_v=1-\frac{s_{k_0,v}}{s_{k_0}(H)}, \  v=1,\ldots, k+1, \\
&z^{(k_0)}_{\sigma_{k_0}} = \frac{1}{s_{k_0}(H)}, \ i=k_0, \ \ \   z^{(i)}_{\sigma_i}=0, \, i\not=k_0, \notag 
\end{align*} 
where for a vertex $v$, $s_{k_0,v}$ is the number of $k_0$-simplices
in $H$ to which $v$ belongs. It is elementary to check that these variables satisfy the
constraints in \eqref{e:dual.cond.0} as equalities. Moreover, it is evident that 
$$
\sum_{i=1}^{k_0}
\sum_{\sigma_i\in H_i}  z^{(i)}_{\sigma_i}=1,
$$
while
\begin{align*}
 &\sum_{v=1}^{k+1}y_v
   =k+1-\frac{1}{s_{k_0}(H)}\sum_{v=1}^{k+1}s_{k_0,v} 
  = k+1-(k_0+1)  =k-k_0,
\end{align*}
since every $k_0$-simplex contributes to exactly $k_0+1$ vertices. Therefore, \eqref{e:small.sum2} reduces to
\begin{equation} \label{e:interm.1}
k-k_0+ q+1-\binom{k}{q}\alpha_q<k+1-\sum_{j=q}^{k_0}s_j(H)\alpha_j.
\end{equation}
Since
$$
s_j(H) \leq {k+1 \choose j+1}, \ \ \  j=q,\ldots, k_0,
$$
\eqref{e:interm.1} follows from  \eqref{e:new.cond}.
\end{proof}

\section{The Betti number at the critical dimension} \label{sec:betti}

In this section we will use the results of Section \ref{sec:simplices}
to derive large deviation results for the Betti number at the critical
dimension. We still assume that the 
probabilities $(p_i, \, i\ge1)$ are given by \eqref{e:power.p}. 
Recall that a dimension $k^*\ge1$ is called critical if
\begin{equation} \label{e:crit.dim}
 \sum_{i=1}^{k^*} {k^*\choose i} \alpha_i<1< \sum_{i=1}^{k^*+1} {k^*+1\choose i}
 \alpha_i, \ \ \text{ and } \ \ q\le k^*; 
\end{equation} 
see \cite{fowler:2019, owada:samorodnitsky:thoppe:2021}. 
We consider the Betti
number $\beta_{k^*, n}$ at the critical dimension.  By the
Morse inequalities (see e.g. Exercise 1, p. 61 in
\cite{munkres:1984}), for any dimension $j$,
\begin{equation} \label{e:morse}
N_n(\sigma_j)-N_n(\sigma_{j-1})-N_n(\sigma_{j+1}) \leq \beta_{j,n}\leq N_n(\sigma_j),
\end{equation}
where $N_n(\sigma_j)$ is the number of $j$-simplices in the
multi-parameter simplicial complex. 
It follows from Proposition 3.1 in
\cite{owada:samorodnitsky:thoppe:2021} that
\begin{equation} \label{e:face.means}
\E\bigl[ N_n(\sigma_j)\bigr] \sim \frac{n^{\tau_j}}{(j+1)!}, \ \ \ 
n\to\infty,
\end{equation}
with
$$
\tau_j= j+1-\sum_{i=1}^j {j+1 \choose i+1}\alpha_i, \ \ j=1,2,\ldots,
$$
and the sequence $\tau_1,\tau_2,\ldots$ reaches its unique maximum at
the critical dimension $k^*$. It thus follows from \eqref{e:morse} and
\eqref{e:face.means} that
\begin{equation*}
\E\bigl[ \beta_{k^*,n}\bigr] \sim \frac{n^{\tau_{k^*}}}{(k^*+1)!}, \ \ \
n\to\infty. 
\end{equation*}

The following theorem considers the upper tail large deviations
for $\beta_{k^*,n}$. An extra condition is needed.
\begin{theorem} \label{prop:betti}
Under the conditions of Proposition \ref{prop:simplex}, assume also
that
\begin{equation} \label{e:new.cond2}
  \sum_{i=q}^{k^*+1}\binom{k^*+1}{i}\alpha_i-1\ge \frac{q(k^*+2)!}{(q+1)!(k^*+1-q)!(k^*+1)} \, \alpha_q. 
\end{equation}
Then for any $\vep>0$, there exist positive constants $ C_1(\vep),\,
C_2(\vep)>0$, such that  
\begin{equation} \label{e:LD.betti}
\exp\Big\{ - C_1(\vep)
             n^{q+1-\binom{k^*}{q}\alpha_q}\log n\Big\}  \leq \P\left(
             \beta_{k^*,n}\ge (1+\vep) \frac{n^{\tau_{k^*}}}{(k^*+1)!}\right) \leq \exp\Big\{
-C_2(\vep) n^{q+1-\binom{k^*}{q}\alpha_q}\Big\}
\end{equation}
for all $n$ large enough. 
\end{theorem}
\begin{proof}
Using  \eqref{e:morse} and Proposition  \ref{prop:simplex}, we have 
\begin{align*}
&\P\left( \beta_{k^*,n}\ge(1+\vep) \frac{n^{\tau_{k^*}}}{(k^*+1)!}\right) \leq
  \P\left( N_n(\sigma_{k^*}) \ge(1+\vep)\frac{n^{\tau_{k^*}}}{(k^*+1)!}\right)
\leq \exp\Bigl\{ -C(\vep) n^{q+1-\binom{k^*}{q}\alpha_q}\Bigr\},
\end{align*}
for some constant $C(\vep)>0$. 
This proves the upper bound in \eqref{e:LD.betti}. 

To prove the lower bound, let  $D$ be a positive constant to be determined in a
moment and take 
\begin{equation*}
  m = \Bigl\lceil ( D(1+\vep))^{1/(k^*+1)} n^{1-\frac{1}{q+1}\binom{k^*}{q}\alpha_q}\Bigr\rceil.
\end{equation*}
Fix arbitrary $m$ vertices out of $n$ (say, the vertices numbered
$1,\ldots, m$). 
Let  $A$ be the event for which all ${m\choose q+1}$
 simplices of dimension $q$ based on these $m$ vertices are present in $K(n;
\bp)$. 
Given a simplicial complex, its simplex  is said to be \emph{free} if it is not in
the boundary of any  simplex of a larger dimension in that complex. 
Define  $F_{k^*,n}$ to be the number of free $k^*$-simplices in $K(n; \bp)$. 
Then, conditionally on $A$, a set of $k^*+1$ vertices in $\{ 1,\dots,m \}$ forms a free $k^*$-simplex in $K(n; \bp)$ with probability 
\begin{equation}  \label{e:prob.free.simplex}
\prod_{i=q+1}^{k^*} p_i^{{k^*+1 \choose i+1}} \Big( 1- \prod_{i=q+1}^{k^*+1}p_i^{\binom{k^*+1}{i}} \Big)^{m-(k^*+1)}
\Big( 1- \prod_{i=q}^{k^*+1}p_i^{\binom{k^*+1}{i}} \Big)^{n-m}. 
\end{equation}
Note that the probability \eqref{e:prob.free.simplex} is independent of the choice of $k^*+1$ vertices in $\{ 1,\dots,m \}$. 
Since there are ${m\choose k^*+1}$ such potential free $k^*$-simplices in $\{ 1,\dots,m \}$, it follows from 
Lemma 3.3 in \cite{janson:oleszkiewicz:rucinski:2004} that 
\begin{align*}
&\P \bigg( F_{k^*,n} \ge \frac12 {m\choose k^*+1} \prod_{i=q+1}^{k^*} p_i^{{k^*+1 \choose
    i+1}}\Big( 1-\prod_{i=q+1}^{k^*+1}p_i^{\binom{k^*+1}{i}}\Big)^{m-(k^*+1)}
    \Big( 1- \prod_{i=q}^{k^*+1}p_i^{\binom{k^*+1}{i}} \Big)^{n-m}\, \bigg|\, A \bigg) \\
&\ge  \frac14 \prod_{i=q+1}^{k^*} p_i^{{k^*+1 \choose i+1}}\Big( 1-\prod_{i=q+1}^{k^*+1}p_i^{\binom{k^*+1}{i}}\Big)^{m-(k^*+1)}
\Big( 1- \prod_{i=q}^{k^*+1}p_i^{\binom{k^*+1}{i}} \Big)^{n-m}. \notag 
\end{align*}
It follows from \eqref{e:new.cond2} and the criticality condition \eqref{e:crit.dim} that 
\begin{align*}
  &\Big( 1-\prod_{i=q+1}^{k^*+1}p_i^{\binom{k^*+1}{i}} \Big)^{m-(k^*+1)}
   \Big( 1- \prod_{i=q}^{k^*+1}p_i^{\binom{k^*+1}{i}} \Big)^{n-m}\\
  \ge &\Big( 1-n^{-\sum_{i=q+1}^{k^*+1}\binom{k^*+1}{i} \alpha_i} \Big)^m
        \Big( 1-n^{-\sum_{i=q}^{k^*+1}\binom{k^*+1}{i} \alpha_i} \Big)^n
\end{align*}
is bounded away from $0$ (by some constant, say, $\rho$). Thus, choosing $D>6/\rho$, one can obtain the following bound: 
\begin{align*}
&\frac12 {m\choose k^*+1} \prod_{i=q+1}^{k^*} p_i^{{k^*+1 \choose
    i+1}}\Big( 1-\prod_{i=q+1}^{k^*+1}p_i^{\binom{k^*+1}{i}}\Big)^{m-(k^*+1)}
    \Big( 1- \prod_{i=q}^{k^*+1}p_i^{\binom{k^*+1}{i}} \Big)^{n-m} \\
&\ge \frac{\rho m^{k^*+1}}{3(k^*+1)!} \, \prod_{i=q+1}^{k^*} p_i^{{k^*+1
  \choose i+1}} 
 \geq  \frac{D\rho (1+\vep)}{3(k^*+1)!}\, n^{\tau_{k^*}}
        >2(1+\vep)\frac{n^{\tau_{k^*}}}{(k^*+1)!}. 
\end{align*}
Now, combining all of these results  gives us 
\begin{align*}
\P\left( F_{k^*,n}\ge 2(1+\vep) \frac{n^{\tau_{k^*}}}{(k^*+1)!}\right) &\ge  \frac{\rho}{4}     \prod_{i=q+1}^{k^*} p_i^{{k^*+1 \choose i+1}} \P(A) =  \frac{\rho}{4}     \prod_{i=q+1}^{k^*} p_i^{{k^*+1 \choose i+1}} p_q^{\binom{m}{q+1}} \\
&=\frac{\rho}{4} \exp\left\{ -\left(\sum_{i=q+1}^{k^*} {k^*+1 \choose i+1}\alpha_i + {m\choose q+1}\alpha_q\right) \log n\right\} \\
&\geq \frac{\rho}{4} \exp\left\{ - C(\vep)  n^{q+1-\binom{k^*}{q}\alpha_q}\log n\right\}, 
\end{align*}
for some constant $ C(\vep)>0$.

Note that the indicator function of each free $k^*$-simplex is an
independent $k^*$-dimensional cocycle. If we denote by $\gamma_{k^*,n}$
the rank of the group of  $k^*$-dimensional cocycles in $K(n; \bp)$, then we have proved that for some constant $ C(\vep)>0$, 
\begin{equation} \label{e:cocycle.rank}
\P\left( \gamma_{k^*,n}\ge 2(1+\vep) \frac{n^{\tau_{k^*}}}{(k^*+1)!}\right) \geq \frac{\rho}{4} \exp\left\{ - C(\vep) n^{q+1-\binom{k^*}{q}\alpha_q}\log n\right\}. 
\end{equation}
In order to  deduce the lower bound in \eqref{e:LD.betti} from 
\eqref{e:cocycle.rank}, we only need to show that the rank of the group
of  $k^*$-dimensional coboundaries in $K(n; \bp)$ exceeds $(1+\vep) n^{\tau_{k^*}}/(k^*+1)!$ only on
an event
whose probability is of  smaller order than the probability in
\eqref{e:cocycle.rank}. The rank of the latter group does not exceed
the rank of the group of  $(k^*-1)$-dimensional cochains in $K(n; \bp)$, which is, of course, equal to the number $N_n(\sigma_{k^*-1})$ of
$(k^*-1)$-simplices in $K(n; \bp)$.  

We know from Proposition 3.1 in \cite{owada:samorodnitsky:thoppe:2021} that
$$
\E [ N_n(\sigma_{k^*-1})]\sim \frac{n^{\tau_{k^*-1}}}{k^*!} = o\bigl(
  n^{\tau_{k^*}}\bigr),  \ \ \ n\to\infty, 
  $$
by the criticality of the dimension $k^*$. It
  remains to notice that for any $\vep>0$,
  \begin{align*}
\P\bigl( N_n(\sigma_{k^*-1})\ge (1+\vep) \E [N_n(\sigma_{k^*-1})]\bigr) \leq \exp\bigl\{
    -Cn^{q+1-\binom{k^*-1}{q}\alpha_q}\bigr\}
  \end{align*}
  by the upper bound in \eqref{e:LDP.simplex}. The right-hand side above exhibits a smaller order
  than the probability in \eqref{e:cocycle.rank}, as required. 
\end{proof}



\end{document}